\documentclass[10pt]{amsart}
\usepackage[english]{babel}
\usepackage{amsmath,amssymb,amsthm,amscd}
\usepackage{url}
\usepackage{color,hyperref}

\usepackage{tikz}
\usetikzlibrary{fit}

\theoremstyle{plain}
\newtheorem{theorem}{Theorem}[section]
\newtheorem{proposition}[theorem]{Proposition}
\newtheorem{lemma}[theorem]{Lemma}
\newtheorem{corollary}[theorem]{Corollary}

\theoremstyle{definition}

\newtheorem{remark}[theorem]{Remark}
\newtheorem{example}[theorem]{Example}

\newtheorem{conjecture}{Conjecture}

\numberwithin{equation}{section}

\newcommand{\N}{\mathbb{N}}
\newcommand{\Z}{\mathbb{Z}}

\newcommand{\set}[2]{\{#1\,|\ #2\}}
\newcommand{\sub}{\subseteq}
\newcommand{\tabulkaa}[9]{
\begin{tabular}{c}
#1\\
\begin{tabular}{ccc}
\begin{tabular}[t]{r|cc}%
$+$ & $w$ & $a$ \\\hline
$w$ & #2 & #3 \\
$a$ & #4 & #5\\
\end{tabular}
&&
\begin{tabular}[t]{r|cc}%
$\cdot$ & $w$ & $a$ \\\hline
$w$ & #6 & #7 \\
$a$ & #8 & #9\\
\end{tabular}
\end{tabular}
\end{tabular}
}

\newcommand{\tabulkab}[9]{
\begin{tabular}{c}
#1\\
\begin{tabular}{ccc}
\begin{tabular}[t]{r|cc}%
$+$ & $a$ & $w$ \\\hline
$a$ & #2 & #3 \\
$w$ & #4 & #5\\
\end{tabular}
&&
\begin{tabular}[t]{r|cc}%
$\cdot$ & $a$ & $w$ \\\hline
$a$ & #6 & #7 \\
$w$ & #8 & #9\\
\end{tabular}
\end{tabular}
\end{tabular}
}
\begin{document}
\title[Congruence-simple multiplicatively idempotent semirings]{Congruence-simple multiplicatively idempotent semirings}

 \author[T.~Kepka]{Tom\'{a}\v{s}~Kepka}
 \address{Department of Algebra, Faculty of Mathematics and Physics, Charles University, Sokolovsk\'{a} 83, 186 75 Prague 8, Czech Republic}
 \email{kepka@karlin.mff.cuni.cz}

\author[M.~Korbel\'a\v{r}]{Miroslav~Korbel\'a\v{r}}
\address{Department of Mathematics, Faculty of Electrical Engineering, Czech Technical University in Prague, Technick\'{a} 2, 166 27 Prague 6, Czech Republic}
\email{miroslav.korbelar@gmail.com}

\author[G.~Landsmann]{\textsc{G\"{u}nter Landsmann}}
\address{Research Institute for Symbolic Computation, Johannes Kepler University, Alten\-bergerstr. 69, A-4040 Linz, Austria}
\email{landsmann@risc.jku.at}


\thanks{
The second and third authors acknowledge the support by the bilateral Austrian Science Fund (FWF) project I 4579-N and Czech Science Foundation (GA\v CR) project 20-09869L ``The many facets of orthomodularity''}

\date{\today}


\begin{abstract}
Let $S$ be a multiplicatively idempotent congruence-simple  semiring. We show that $|S|=2$ if $S$ has a multiplicatively absorbing element. We also prove that if  $S$ is finite then either  $|S|=2$ or $S\cong End(L)$ or $S^{op}\cong End(L)$ where $L$ is a 2-element semilattice. It seems to be an open question, whether $S$ can be infinite at all.
\end{abstract}

\maketitle



The class of finite (congruence-)simple lattices is quite opulent. Even in the case of the modular ones, there are infinitely many (non-isomorphic) examples of such simple lattices  (see e.g. \cite{graetzer,schmidt}). Of course, a distribute lattice is simple if and only if it is a two-element chain. Now, distributive lattices form an (equational) subclass of the class $\mathcal{M}$ of all multiplicatively idempotent semirings. 

The main aim of this paper is to show that (up to isomorphism) there exist exactly eight finite congruence-simple  semirings in $\mathcal{M}$ (six of them are two-element and the remaining two are three-element). We also prove that every semiring in $\mathcal{M}$ possessing a multiplicatively absorbing element is finite (and has only two elements). It seems to be an open question, whether there are infinite congruence-simple semirings in $\mathcal{M}$ at all.

On the other hand, one can find infinite semirings in $\mathcal{M}$ that are bi-ideal-simple (ideal-simple, resp.). Moreover, there are infinitely many (non-isomorphic) examples of finite  semirings in $\mathcal{M}$ that are bi-ideal-simple (ideal-simple, resp.) (see, e.g., Examples \ref{ex_1} and \ref{ex_3}).


\section{Preliminaries}

A \emph{semiring} $S$ is a non-empty set equipped with two associative binary operations (usually denoted as addition and multiplication) such that the addition is commutative and the multiplication distributes over the addition from both sides. The semiring $S$ is called \emph{commutative} if the multiplication is commutative.

For non-empty subsets $A,B\sub S$ we will use the usual notation of their sum and product as $A+B=\set{a+b}{a\in A, b\in B}$ and $AB=\set{a\cdot b}{a\in A, b\in B}$. If $A=\{a\}$ for some $a\in S$, we sometimes omit the brackets for a simpler notation.

A non-empty subset $I$ of $S$ is an \emph{ideal} (\emph{bi-ideal}, resp.) of $S$ if $(I+I)\cup SI\cup IS\sub I$ ($(S+I)\cup SI\cup IS\sub I$, resp.). A semiring $S$ is called \emph{congruence-simple} if $S$ has just two congruences and \emph{ideal-simple} (\emph{bi-ideal-simple}, resp.) if $|S|\geq 2$ and $I=S$ whenever $I$ is an ideal (bi-ideal, resp.) containing at least two elements. 

A semiring $S$ is called \emph{multiplicatively (additively, resp.) idempotent} if $x^2=x$ ($x+x=x$, resp.) for every $x\in S$. A semiring that is both additively and multiplicatively idempotent will be called \emph{bi-idempotent}. 


An element $w\in S$ is called 
\begin{itemize}
\item \emph{right (left, resp.) multiplicatively  absorbing} if $Sw=\{w\}$ ($wS=\{w\}$, resp.);
\item \emph{multiplicatively absorbing} if $w$ is both right and left multiplicatively absorbing;
\item \emph{additively absorbing} if $S+w=\{w\}$;
\item \emph{multiplicatively (additively, resp.) neutral} if $xw=x=wx$ ($x+w=x$, resp.) for every $x\in S$; 
\item \emph{bi-absorbing} if it is both multiplicatively and additively absorbing (such an element will be denoted by $o_S$);
\item a \emph{zero} if it is multiplicatively absorbing and additively neutral (such an element will be denoted by $0_S$).
\end{itemize}

For a semiring $S(+,\cdot)$ the opposite semiring $S^{op}(+,\ast)$ is defined as $a\ast b=b\cdot a$ for every $a,b\in S^{op}=S$.

For a semiring $S$ with a multiplicatively absorbing element $w\in S$ and $T=S\setminus\{w\}$ let us denote by $\varrho_S=(T\times T)\cup\{(w,w)\}$ an equivalence on $S$.


\begin{remark}\label{remark_1}
Let $S$ be a semiring.

(i) Clearly, if $S$ is ideal-simple, then $S$ is bi-ideal-simple. Also, if $S$ is congruence-simple,  then  $S$ is bi-ideal-simple as well.

Indeed, if $I$ is a bi-ideal of $S$ and $|I|\geq 2$ then the equivalence relation $\varrho=(I\times I) \cup \set{(a,a)}{a\in S}$ is a congruence on $S$ such that $\varrho\neq id_S$. If $S$ is congruence-simple, we obtain that $\varrho=S\times S$ and therefore  $I=S$.

(ii) Let $S$ have a multiplicatively absorbing element $w\in S$. Let $\varrho$ be a congruence on $S$ and $I=\set{a\in S}{(a,w)\in\varrho}$ be the block of this congruence. Clearly, $I=S$ if and only if $\varrho=S\times S$. If $w=o_S$ then $I$ is a bi-ideal of $S$ and, similarly, if $w=0_S$ then $I$ is an ideal of $S$.


(iii) Let $S$ have a bi-absorbing element $o_S\in S$ and $T=S\setminus\{o_S\}\neq\emptyset$. Then $\varrho_S$ is an equivalence defined on $S$ and $id_S\neq\varrho_S\neq S\times S$ if and only if $|S|\geq 3$. Besides, $\varrho_S$ is a congruence on the semiring $S$ in each of the following cases:
\begin{enumerate}
 \item[(1)] $T+T=\{o_S\}$, $TT=\{o_S\}$;
 \item[(2)] $T+T=\{o_S\}$, $TT\sub T$;
 \item[(3)] $T+T\sub T$, $TT=\{o_S\}$;
 \item[(4)] $T+T\sub T$, $TT\sub T$.
\end{enumerate}

(iv) If the semiring $S$ is non-trivial and $0_S\in S$ then $S$ is bi-ideal-simple (in fact, the only bi-ideal of $S$ is $S$ itself).

(v) Every two-element semiring is both congruence-simple and ideal-simple.

(vi) Let $S$ be multiplicatively idempotent and additively cancellative (i.e., $a+c\neq b+c$ for all $a,b,c\in S$ such that $a\neq b$). We claim that $S$ is a Boolean ring.

Indeed, as $S$ is additively cancellative, $S$ is a subsemiring of some ring $R$. For every $a,b\in S$ we have $a+b=(a+b)^2=a^2+b^2+ab+ba=a+b+ab+ba$ and therefore $0_R=ab+ba\in S$. In particular $0_S=0_R\in S$ and we obtain that $0_S=a^2+a^2=a+a$ for every $a\in S$. Therefore $S$ is a ring of characteristic $2$. Finally, we have $ab=-ba=ba$ for all $a,b\in S$ and $S$ is thus commutative. Henceforth, $S$ is a Boolean ring.

If, moreover, the (semi)ring $S$ is congruence-simple or ideal-simple then $S$ is a field and, by the multiplicative idempotency, $S$ is isomorphic to the two-element field $\Z_2$. 

(vii) \cite[Theorem 2.2]{cornish} Let $S$ be congruence-simple and $0_S\in S$. If $a^2\neq 0_S$ for every $a\in S$, $a\neq 0_S$, then $S$ has no proper divisors of zero. 
\end{remark}

\begin{remark}\label{idempotent}
 Let $B$ be a band  (i.e., an idempotent semigroup). If $b\in BaB$ for some $a\in B$ then $b=bab$.
 
Indeed, if $b=cad$ for some $c,d\in B$. Then we have $b=cad=c(ad)^2=(cad)ad=bad$ and, similarly, $b=bad=(ba)^2d=ba(bad)=bab$. 

Notice that the set $BaB$ is the principal ideal of the semigroup $B$ generated by the element $a$. 
If the band $B$ is ideal-simple, then $b=bab$ for all $a,b\in B$ (i.e., the band is rectangular).
\end{remark}

\section{Bi-ideal-simple multiplicatively idempotent semirings with a  multiplicatively absorbing element}

Throughout this section, let $S$ be a multiplicatively idempotent bi-ideal-simple semiring possessing a multiplicatively absorbing element $w$. We put $T=S\setminus\{w\}$.

\begin{lemma}\label{2.1}
 Put $A(S)=\set{a\in S}{SaS+S=\{w\}}$ and $B(S)=\set{a\in S}{SaS+S=S}$. Then:
 \begin{enumerate}
  \item[(i)] $2w=w$ and either $w=0_S$ or $w=o_S$.
  \item[(ii)] $B(S)=S\setminus A(S)$.
  \item[(iii)] Either $A(S)=\emptyset$ or $A(S)$ is a bi-ideal of $S$.  
  \item[(iv)] $2a=w$ for every $a\in A(S)$.
 \end{enumerate}
\end{lemma}
\begin{proof}
(i) We have $2w = w + w = w^2 + w^2 =
w(w + w) = w$. Further, the set $S+w$ is  a bi-ideal of $S$. If $|S+w|=1$ then $w=o_S$. Assume that $|S+w|\geq 2$. Since $S$ is bi-ideal-simple, we have $S+w=S$. Then for every $a\in S$ there is $b\in S$ such that $b+w=a$. Hence  $a+w=b+w+w=b+w=a$ and therefore $w=0_S$.
  
  (ii) For every $a\in S$ the set $SaS+S$ is a bi-ideal of $S$ and $w\in SaS+S$. As $S$ is bi-ideal simple, it follows that either $|SaS+S|=1$ or $SaS+S=S$. The rest is obvious.
 
 (iii) Let $a\in A(S)$ and $s\in S$. Then $S(a+s)S+S\sub SaS+S=\{w\}$,  $S(as)S+S\sub SaS+S=\{w\}$ and $S(sa)S+S\sub SaS+S=\{w\}$. Hence $s+a,sa,as\in A(S)$ and $A(S)$ is a bi-ideal of $S$.
 
 (iv) For every $a\in A(S)$ we have $2a=a^3+a\in SaS+S=\{w\}$.
\end{proof}

\begin{proposition}\label{2.1.1}
Assume that  $|S|=2$. Then $S$ is isomorphic to exactly one of the following four commutative semirings:

\begin{center}
\begin{tabular}[t]{cc}%
\tabulkaa{$\mathbb{S}_1$}{$w$}{$a$}{$a$}{$w$}{$w$}{$w$}{$w$}{$a$} &
\tabulkaa{$\mathbb{S}_2$}{$w$}{$a$}{$a$}{$a$}{$w$}{$w$}{$w$}{$a$} \\
&\\[6pt]
\tabulkaa{$\mathbb{S}_3$}{$w$}{$w$}{$w$}{$a$}{$w$}{$w$}{$w$}{$a$} &
\tabulkaa{$\mathbb{S}_4$}{$w$}{$w$}{$w$}{$w$}{$w$}{$w$}{$w$}{$a$}\\
\end{tabular}
\end{center}
\bigskip

The semirings $\mathbb{S}_1$, $\mathbb{S}_2$ have a zero element $w=0_S$ and the semirings $\mathbb{S}_3$, $\mathbb{S}_4$ have a bi-absorbing element $w=o_S$. The only case that is a ring is $\mathbb{S}_1$.
\end{proposition}
\begin{proof}
 It is easy to verify that $\mathbb{S}_1$, $\mathbb{S}_2$, $\mathbb{S}_3$ and $\mathbb{S}_4$ are semirings. Let $S=\{w,a\}$, where $a\neq w$. By Lemma \ref{2.1}(i), either $w=0_S$ or $w=o_S$. Since $S$ is multiplicatively idempotent, the multiplication is fully determined. The addition is also determined  up to the case $a+a\in\{a,w\}$ (with $w$ being either a zero or a bi-absorbing element). All these four possibilities are represented by the cases $\mathbb{S}_1$, $\mathbb{S}_2$, $\mathbb{S}_3$ and $\mathbb{S}_4$.
\end{proof}

\begin{lemma}\label{2.4}
 Assume that $S+S=\{w\}$. Then $w=o_S$, the relation $\varrho_S$ is a congruence on $S$ and $S/\varrho_S\cong\mathbb{S}_4$.
\end{lemma}
\begin{proof}
As $S+S=\{w\}$, the element $w$ is  bi-absorbing, i.e., $w=o_S$. For every $a\in T= S\setminus\{o_S\}$, we have $o_S\in SaS$. Hence  the set $SaS$ is a bi-ideal of the semiring $S$ and $|SaS|\geq 2$. Since the semiring $S$ is is bi-ideal-simple, we get $SaS=S$ and, consequently $b=bab$ for every $b\in T$, by Remark \ref{idempotent}.  It follows that $o_S\neq b=b(ab)$ and therefore $ab\neq o_S$. That is, $TT\sub T$.  Now, it follows from Remark \ref{remark_1}(iii)(2) that $\varrho_S$ is a congruence on the semiring $S$. As $|S/\varrho_S|=2$, we see readily that $S/\varrho_S\cong \mathbb{S}_4$, by Proposition \ref{2.1.1}.
\end{proof}

\begin{lemma}\label{2.6}
 Assume that $S$ is additively idempotent and $w=o_S$. Then $TT\sub T$ and $T+T\sub T$. 
\end{lemma}
\begin{proof}
 Since $S$ is additively idempotent, it follows from Lemma \ref{2.1}(iv) and (ii), that $A(S)=\{o_S\}$ and $B(S)=S\setminus\{o_S\}=T$. Then, for $a\in T=B(S)$, we obtain that $SaS+S=S$. Let $b\in T$. Since $b\in S=SaS+S$, there are $x,y,z\in S$ such that $b=c+z$ where $c=xay\in SaS$. By Remark \ref{idempotent}, we have that $c=cac$. Now, $b=c+c+z=c+b=b+cac$.
 
 Assume now, for contrary that $a+b=o_S$. By the previous part of the proof, $b=b+cac$ for some $c\in S$. Hence  we have $cbc=c(b+cac)c=cbc+cac=c(a+b)c=o_S$. It follows that $cb=cb\cdot b=cb(b+cac)=cb+(cbc)ac=cb+o_S=o_S$. Finally, we obtain that $b=b\cdot b=(b+cac)b=b+ca(cb)=b+o_S=o_S$, a contradiction. We have shown that $a+b\in S\setminus\{o_S\}=T$. Therefore $T+T\sub T$.

Finally, from $b=b+cac$ it follows that $b=b\cdot b\cdot b=b(b+cac)b=b+bcacb=b+e$ where $e=bcacb\in SaS$. By Remark \ref{idempotent}, $e=eae$ and $b=b+eae$. In particular, we obtain that $b=b+eae=b+ea(bcacb)$. Now if $ab=o_S$, then $b=b+e(ab)cacb=b+o_S=o_S$, a contradiction. Hence we have shown that $ab\in S\setminus\{o_S\}=T$. Therefore $TT\sub T$.
\end{proof}

\begin{lemma}\label{bi-absorbing}
  Assume that $S$ is additively idempotent and $w=o_S$. Then $\varrho_S$ is a congruence on $S$ and $S/\varrho_S\cong\mathbb{S}_3$.
\end{lemma}
\begin{proof}
 By Lemma \ref{2.6}, we have that $T+T, TT\sub T$. Hence the relation $\varrho_S=(T\times T)\cup \{(o_S,o_S)\}$ is a congruence on the semiring $S$, by Remark \ref{remark_1}(iii)(4). As $|S/\varrho_S|=2$, we obtain that $S/\varrho_S\cong \mathbb{S}_3$, by Proposition \ref{2.1.1}.
\end{proof}

\begin{lemma}\label{2.7}
 Assume that $S$ is additively idempotent and  $w=0_S$.  Then $\varrho_S$ is a congruence on $S$ and $S/\varrho_S\cong\mathbb{S}_2$.
\end{lemma}
\begin{proof}
 If $a,b\in S$ are such that $a+b=0_S$, then $a=a+0_S=a+a+b=a+b=0_S$ and, similarly, $b=0_S$. Henceforth $T+T\sub T$. Since $S$ is multiplicatively idempotent, it follows that $a^2=a\neq 0_S$ for every $a\in S$, $a\neq 0_S$. By Remark \ref{remark_1}(vii), we have $TT\sub T$. Hence, the relation $\varrho_S=(T\times T)\cup\{(0_S,0_S)\}$ is a congruence on the semiring $S$, by Remark \ref{remark_1}(iii)(4). Since $|S/\varrho_S|=2$, we obtain that $S/\varrho_S\cong \mathbb{S}_2$, by Proposition \ref{2.1.1}.
\end{proof}

\begin{proposition}
 Assume that either $S+S=\{w\}$ or $S$ is additively idempotent.  Then $\varrho_S$ is a congruence on $S$ and $S/\varrho_S$ is isomorphic to (just) one of the two-element semirings $\mathbb{S}_2$, $\mathbb{S}_3$, $\mathbb{S}_4$.
\end{proposition}
\begin{proof}
 Combine \ref{2.4}, \ref{bi-absorbing} and \ref{2.7}.
\end{proof}

The following examples show that there are infinitely many finite  multiplicatively idempotent semirings that are \emph{bi-ideal-simple} (Examples \ref{ex_1} and \ref{2.10}) but neither congruence- nor ideal-simple. Similarly, there are infinitely many finite  multiplicatively idempotent semirings that are \emph{ideal-simple} (Example \ref{ex_3}) but not congruence-simple.

\begin{example}\label{ex_1} 
 Assume that $w=o_S\in S$ ($w=0_S\in S$, resp.) for the multiplicatively idempotent bi-ideal-simple semiring $S$. Now, set $P=S\cup\{z\}$, $z\not\in S$, $z=0_P$ ($z=o_P$, resp.). In this way, $P$ becomes a multiplicatively idempotent bi-ideal-simple semiring with a zero (a bi-absorbing element, resp.). Clearly, $|P|=|S|+1$, $P$ is bi-idempotent iff $S$ is so, and $P$ is commutative iff $S$ is so. On the other hand, $P$ is \emph{not} ideal-simple as the two-element set $\{z,w\}$ is always an ideal of $P$. Also, $P$ is \emph{not} congruence-simple, as, by Lemmas \ref{bi-absorbing} and \ref{2.7}, $P/\varrho_P\cong\mathbb{S}_2$ ($P/\varrho_P\cong\mathbb{S}_3$, resp.). Of course, we can start with the two-element semirings $S=\mathbb{S}_3,\mathbb{S}_4$ ($S=\mathbb{S}_1,\mathbb{S}_2$, resp.) and, continuing by means of the ``zig-zag method``, we arrive at examples of any finite size $\geq 3$ (see also \cite{vechtomov-petrov}).
\end{example}

\begin{example}\label{2.10}
The following five-element semiring $\mathbb{P}$ is constructed in \cite[Theorem 2.1]{cornish}:

\bigskip

\begin{center}
 \begin{tabular}{c}
$\mathbb{P}$\\
\begin{tabular}{ccc}
\begin{tabular}[t]{r|ccccc}%
$+$ & $0$ & $1$ & $a$ & $b$ & $c$ \\\hline
$0$ & $0$ & $1$ & $a$ & $b$ & $c$  \\
$1$ & $1$ & $1$ & $c$ & $c$ & $c$  \\
$a$ & $a$ & $c$ & $a$ & $c$ & $c$  \\
$b$ & $b$ & $c$ & $c$ & $b$ & $c$  \\
$c$ & $c$ & $c$ & $c$ & $c$ & $c$  \\
\end{tabular}
&&
\begin{tabular}[t]{r|ccccc}%
$\cdot$ & $0$ & $1$ & $a$ & $b$ & $c$ \\\hline
$0$ & $0$ & $0$ & $0$ & $0$ & $0$  \\
$1$ & $0$ & $1$ & $a$ & $b$ & $c$  \\
$a$ & $0$ & $a$ & $a$ & $0$ & $a$  \\
$b$ & $0$ & $b$ & $0$ & $b$ & $b$  \\
$c$ & $0$ & $c$ & $a$ & $b$ & $c$  \\
\end{tabular}
\end{tabular}
\end{tabular} 
\end{center} 

\bigskip
 
The semiring $\mathbb{P}$ is commutative, bi-idempotent and has a zero element $0=0_\mathbb{P}$. This semiring  is subdirectly irreducible (where $\sigma=(J\times J)\cup id_\mathbb{P}$ for $J=\{c,1\}$ is the smallest congruence such that $\sigma\neq id_\mathbb{P}$), but neither congruence-simple  nor ideal-simple (e.g., the set $I=\{0,a\}$ is an ideal of $\mathbb{P}$). On the other hand, it is easy to verify that $\mathbb{P}$ is bi-ideal-simple.  

\end{example}

\begin{example}\label{ex_3}
 Let $L(+)$ be a non-trivial semilattice. Define a multiplication on $L$ by $ab=b$ ($ab=a$, resp.) for all $a,b\in L$. Then $L(+,\cdot)$ becomes an ideal-simple bi-idempotent semiring. By a latter result in this paper (Theorem \ref{3.3}), if $|L|\geq 4$ then $L$ cannot be congruence-simple. 
\end{example}

\begin{remark}\label{remark_5}
Let $P$ be a \emph{commutative} multiplicatively idempotent semiring.

(i) If $P$ is ideal-simple then, according to the classification in \cite[Theorem 11.2]{simple_comm}, $P$ has  just two elements (and is isomorphic to one of the semirings $\mathbb{S}_1$, $\mathbb{S}_2$, $\mathbb{S}_3$, $\mathbb{S}_4$).

(ii) If $P$ is congruence-simple then it follows easily from \cite[Theorem 10.1]{simple_comm} that $P$ is also isomorphic to $\mathbb{S}_1$, $\mathbb{S}_2$, $\mathbb{S}_3$ and $\mathbb{S}_4$ (see also \cite[Lemma 3.1]{vechtomov-petrov}).

(iii) If $P$ is subdirectly irreducible then $P$ is bi-ideal-simple (and if $P$ has at least three elements then, by (i) and (ii), it is nor congruence-simple neither ideal-simple).

Indeed, by \cite[Theorem 3.1]{vechtomov-petrov}, $P$ has a unity $1\in P$ and there is $e\in P\setminus\{1\}$ such that for the smallest non-identical congruence $\tau$ on $P$ is $(1,e)\in\tau$. Now, let $I$ be a bi-ideal of $P$ such that $|I|\geq 2$. Then $\varrho=(I\times I)\cup\set{(a,a)}{a\in P}$ is a congruence on $P$ and $\varrho\neq id_P$. Hence $(1,e)\in\tau\sub\varrho$. Therefore $1\in I$ and we have $I=P$. 
\end{remark}

\section{Congruence-simple multiplicatively idempotent semirings with a multiplicatively absorbing element}



%


\begin{proposition}\label{3.1}
 Let $S$ be a multiplicatively idempotent congruence-simple semiring containing at least three elements. Then $S$ is additively idempotent (so that $S$ is bi-idempotent).
\end{proposition}
\begin{proof}
Firstly, the set $S+S$ is a bi-ideal of $S$. As the semiring $S$ is bi-ideal-simple, it follows that either $|S+S|=1$ or that $S+S=S$. 
 If $|S+S|=1$ then, by Lemma \ref{2.4}, there is a congruence $\varrho\neq id_S$ on the semiring $S$ such that $|S/\varrho|=2$, a contradiction with the congruence-simpleness of $S$. Hence $S+S=S$.

 Further, consider the equivalence $\sigma$ on $S$ defined as $(x,y)\in\sigma$ if and only if $2x=2y$. This equivalence, clearly, is a congruence on the semiring $S$. 
 
 We show that $\sigma=id_S$. Assume, for contrary, that $\sigma\neq id_S$. Since $S$ is congruence-simple, it follows that $\sigma=S\times S$. Hence there is $w\in S$ such that $2x=w$ for every $x\in S$. Obviously, the element $w$ is multiplicatively absorbing and therefore, by Lemma \ref{2.1}(i), we have either $w=0_S$ or $w=o_S$. 
If $w=0_S$ then $S$ is a ring and, by Remark  \ref{remark_1}(vi), we obtain that $S\cong\Z_2$, a contradiction with $|S|\geq 3$. 

Therefore we have $w=o_S$. Now, as $S+S=S$, for every $a\in S$ there are $b,c\in S$ such that $a=b+c$. Hence $ab+ca=(b+c)b+c(b+c)=2cb+b^2+c^2=o_S+b+c=o_S$. It follows that $a=a\cdot a\cdot a=a(b+c)a=aba+aca=a\cdot ab\cdot a+a\cdot ca\cdot a=a(ab+ca)a=ao_Sa=o_S$. We have obtained that $|S|=1$, a final contradiction.
 
Thus, we have shown that $\sigma=id_S$. This means that for every $x,y\in S$ the condition $2x=2y$ implies that $x=y$. Now, for every $a\in S$ it holds that $2a=(2a)^2=4a^2=2(2a)$. It follows that $a=2a$ and the semiring $S$ is therefore additively idempotent.    
 \end{proof}

\begin{theorem}\label{2.9}
 Let $S$ be  a multiplicatively idempotent congruence-simple semiring possessing a multiplicatively absorbing element. Then $S$ is isomorphic to one of the two-element semirings $\mathbb{S}_1$, $\mathbb{S}_2$, $\mathbb{S}_3$ or $\mathbb{S}_4$.
\end{theorem}
\begin{proof}
By Lemma \ref{2.1}(i), $S$ has either a bi-absorbing element $o_S$ or a zero element $0_S$. Further, proceeding by contradiction, assume that $|S|\geq 3$. Then, by Proposition \ref{3.1}, $S$ is bi-idempotent. By Remark  \ref{remark_1}(i), the semiring $S$ is bi-ideal-simple. Therefore, by Propositions \ref{bi-absorbing} and \ref{2.7}, we obtain that there is a congruence $\varrho$ on $S$ such that $\varrho\neq id_S$ and $|S/\varrho|=2$. This is a contradiction with the fact that $S$ is congruence-simple.
 
 We may therefore assume that $|S|=2$. The rest follows immediately from Proposition \ref{2.1.1}.
\end{proof}

\begin{corollary}\label{bi-ideal-simple}
 Let $S$ be a multiplicatively idempotent semiring possessing  a bi-absorbing element $w$.  If 
 \begin{enumerate}
 \item either $w=o_S$ and $S$ is bi-ideal-simple 
 \item or $w=0_S$ and $S$ is ideal-simple
 \end{enumerate}
 then $\varrho_S$ is the greatest non-trivial congruence on $S$ (i.e., the congruence $\varrho_S$ is the unique co-atom of the congruence lattice of $S$.)
\end{corollary}
\begin{proof}
 Let $\sigma\neq S\times S$ be a congruence on the semiring $S$. Then there is $a\in T=S\setminus\{w\}$ such that $(a,w)\notin\sigma$. By Zorn's lemma there is a congruence $\sigma'$ of $S$ that is maximal with respect to the property $(a,w)\notin\sigma'$ and $\sigma\sub\sigma'$. We claim that $\sigma'$ is a maximal congruence on $S$.

Indeed, if $\tau$ is a congruence on $S$ such that $\sigma'\sub\tau$ and $\sigma'\neq \tau$, then $a,w\in I=\set{x\in S}{ (x,w)\in\tau}$. By Remark \ref{remark_1}(ii) the set  $I$ is a bi-ideal (an ideal, resp.) of $S$ and $|I|\geq 2$. Since $S$ is bi-ideal-simple  (ideal-simple, resp.), we have that $I=S$ and therefore $\tau=S\times S$.

Similarly, by Remark \ref{remark_1}(ii), the set $J=\set{x\in S}{ (x,w)\in\sigma'}$ is a bi-ideal (an ideal) of $S$. Since the factor-semiring $S/\sigma'$ is non-trivial, we have that $J\neq S$. As the semiring $S$ is bi-ideal-simple  (ideal-simple, resp.), it follows that $J=\{w\}$. Now, since $\sigma'$ is a maximal congruence, we obtain, by Theorem \ref{2.9}, that the semiring $S/\sigma'$ has precisely two elements. Therefore the set $T=S\setminus\{w\}$ is a block of the congruence $\sigma'$ and $\sigma\sub\sigma'=(T\times T)\cup\{(w,w)\}=\varrho_S$. 
It means that $\varrho_S$ is the only maximal congruence on the semiring $S$ and every proper congruence on $S$ is contained in $\varrho_S$. 
\end{proof}

\begin{remark}\label{3.4.0} 
Let $S$ be a semiring as in Corollary \ref{bi-ideal-simple} fulfilling the condition (1) and $T=S\setminus\{o_S\}$.
 
 If $S/\varrho_S\cong\mathbb{S}_3$, then, by the definition of $\mathbb{S}_3$, the set $T$ is a subsemiring of $S$. If $|T|\geq 2$ then $T$ is bi-ideal-simple and has no bi-absorbing element (otherwise, for $o_T\in T$ the set $\{o_S, o_T\}$ is a bi-ideal of $S$, a contradiction).
 
 If $S/\varrho_S\cong\mathbb{S}_4$, then $S+S=\{o_S\}$ and the semiring $S$ is ideal-simple. The band $S(\cdot)$ is ideal-simple as well. Besides, $TT\sub T$ and the band $T(\cdot)$ is ideal-simple and rectangular (see Remark \ref{idempotent}). If $|T|\geq 2$ then the semigroup $T(\cdot)$ has no multiplicatively absorbing element (otherwise, for such an element $w'\in T$ the set $\{o_S, w'\}$ is a bi-ideal of $S$, a contradiction).
 \end{remark}

\begin{remark}\label{3.5}
Let $S$ be a semiring as in Corollary \ref{bi-ideal-simple} fulfilling the condition (2) and $T=S\setminus\{0_S\}$.

If $S/\varrho_S\cong\mathbb{S}_1$ then $S\cong\mathbb{S}_1$. 
Indeed, we have that $T+T\sub\{0_S\}$.  Hence for all $a,b\in T$ it holds that $a+a=0_S=a+b$. Therefore $a=a+0_S=a+a+b=0_S+b=b$ and it follows that $|T|=1$ and $S\cong\mathbb{S}_1$.

If $S/\varrho_S\cong\mathbb{S}_2$ then the set  $T$ is a subsemiring of $S$. If $|T|\geq 2$ then $T$ is ideal-simple and has no multiplicatively absorbing element (otherwise, for such an element $w'\in T$ the set $\{0_S, w'\}$ is an ideal of $S$, a contradiction).
 \end{remark}

\section{Finite congruence-simple multiplicatively idempotent semirings}

For our further consideration let $L(+)$ be a finite semilattice with the greatest element $1$. Denote by $End(L)(+,\cdot)$ the semiring of all endomorphisms of the semilattice $L$. For $\varphi,\psi\in End(L)$ the operations are defined as follows $(\varphi+\psi)(x)=\varphi(x)+\psi(x)$ and $(\varphi\cdot\psi)(x)=\varphi(\psi(x))$ for every $x\in L$.

If, moreover, $L$ has the least element $0$ (i.e., $L$ is a lattice), we set $End_0(L)=\set{\varphi\in End(L)}{\varphi(0)=0}$. Clearly, $End_0(L)$ is a subsemiring of $End(L)$.

\begin{proposition}\label{3.0}
 Let $S$ be a multiplicatively idempotent semiring without a multiplicatively absorbing element. If $|S|=2$ then $S$ is bi-idempotent and isomorphic to precisely one of the following semirings $\mathbb{S}_5$ or $\mathbb{S}_6$. Moreover, $\mathbb{S}_6^{op}= \mathbb{S}_5$.

\bigskip

\begin{center}
\begin{tabular}[t]{cc}%
\tabulkab{$\mathbb{S}_5$}{$a$}{$w$}{$w$}{$w$}{$a$}{$a$}{$w$}{$w$} &
\tabulkab{$\mathbb{S}_6$}{$a$}{$w$}{$w$}{$w$}{$a$}{$w$}{$a$}{$w$}\\
\end{tabular}
\end{center}
\end{proposition}
\begin{proof}
Let $S=\{a,w\}$ and $a\neq w$. Assume, for contrary, that $a+a\neq a$. Then $w=a+a=a^2+a^2=a(a+a)=aw$ and, similarly, $w=wa$. Since $w=w^2$, the element $w$ is multiplicatively absorbing, a contradiction. Hence $a+a=a$ and, analogously, $w+w=w$. Therefore, $S$ is additively idempotent.  

Further, we may assume without loss of generality that $a<w$. Both the operations in $S$ are now determined up to the case $a\cdot w$ and $w\cdot a$. Since $S$ has no multiplicatively absorbing element, it follows that $a\cdot w\neq w\cdot a$. All the possibilities are now represented by the cases $\mathbb{S}_5$ and $\mathbb{S}_6$. 
\end{proof}

The following assertion is easy to verify.

\begin{proposition}\label{3.2}
Let $L=\{0,1\}$ be a semilattice with $0<1$. Then $End(L)=\{a,b,w\}$, where $a(x)=0$, $b(x)=x$ and $w(x)=1$ for every $x\in L$. 

The semirings  $\mathbb{S}_7=End(L)$ and $\mathbb{S}_8=End(L)^{op}$ have three elements, are bi-idempotent, congruence-simple  and are without a multiplicatively absorbing element. These two semirings are non-isomorphic and also not ideal-simple. The operations are as follows.

\bigskip

\begin{center}
 \begin{tabular}{c}
$\mathbb{S}_7$\\
\begin{tabular}{ccc}
\begin{tabular}[t]{r|ccc}%
$+$ & $a$ & $b$ & $w$ \\\hline
$a$ & $a$ & $b$ & $w$  \\
$b$ & $b$ & $b$ & $w$  \\
$w$ & $w$ & $w$ & $w$  \\
\end{tabular}
&&
\begin{tabular}[t]{r|ccc}%
$\cdot$ & $a$ & $b$ & $w$ \\\hline
$a$ & $a$ & $a$ & $a$  \\
$b$ & $a$ & $b$ & $w$  \\
$w$ & $w$ & $w$ & $w$  \\
\end{tabular}
\end{tabular}
\end{tabular} 
\end{center}

\bigskip

\begin{center}
 \begin{tabular}{c}
$\mathbb{S}_8$\\
\begin{tabular}{ccc}
\begin{tabular}[t]{r|ccc}%
$+$ & $a$ & $b$ & $w$ \\\hline
$a$ & $a$ & $b$ & $w$  \\
$b$ & $b$ & $b$ & $w$  \\
$w$ & $w$ & $w$ & $w$  \\
\end{tabular}
&&
\begin{tabular}[t]{r|ccc}%
$\cdot$ & $a$ & $b$ & $w$ \\\hline
$a$ & $a$ & $a$ & $w$  \\
$b$ & $a$ & $b$ & $w$  \\
$w$ & $a$ & $w$ & $w$  \\
\end{tabular}
\end{tabular}
\end{tabular} 
\end{center} 

\bigskip

\end{proposition}

\begin{remark}
 Notice that the semirings $\mathbb{S}_5$, $\mathbb{S}_6$, $\mathbb{S}_7$ and $\mathbb{S}_8$ have an additively neutral element $a$. Besides, the semirings $\mathbb{S}_7$ and $\mathbb{S}_8$ have a multiplicatively neutral element $b$ and none of these two semirings is ideal-simple. Finally, notice that the element $w$ is left multiplicatively absorbing in $\mathbb{S}_5$ and $\mathbb{S}_7$ and right multiplicatively absorbing in $\mathbb{S}_6$ and $\mathbb{S}_8$.
\end{remark}

 \begin{theorem}\cite[Theorems 5.1 and 5.3]{zumbragel}\label{classification}
 Let $S$ be a finite bi-idempotent congruence-simple semiring with the greatest element $w\in S$ (with respect to the addition) and $|S|\geq 3$. 
 \begin{enumerate}
  \item[(i)] Let $w$ be neither left nor right multiplicatively absorbing. Then there is a finite (semi)lattice $L$ (with the least element $0$ and the greatest element $1$) such that $S$ is isomorphic to a subsemiring $R$ of $End_{0}(L)$. Further, for every $a,b\in L$ there is $e_{a,b}\in R$ such that  for every $x\in L$ is $e_{a,b}(x)=0$ if $x\leq a$ and $e_{a,b}(x)=b$ otherwise.
  \item[(ii)] Let $w$ be left but not right multiplicatively absorbing. Then there is a finite semilattice $L$ (with the greatest element $1$) such that $S$ is isomorphic to a subsemiring $R$ of $End(L)$. Further, for every $a\in L$ and $b\in L\setminus\{1\}$ there is $f\in R$ such that for every $x\in L$ is $f(x)=b$ if $x\leq a$ and $f(x)>b$ otherwise.
 \end{enumerate}
\end{theorem}

\begin{proposition}\label{non-bi-absorbing}
 Let $S$ be a finite bi-idempotent congruence-simple semiring with the greatest element $w\in S$ (with respect to the addition). If $w$ is not a bi-absorbing element then $S$ is isomorphic to one of the four two/three-element semirings $\mathbb{S}_5$, $\mathbb{S}_6$, $\mathbb{S}_7$ or $\mathbb{S}_8$.
\end{proposition}
\begin{proof}
If $|S|=2$ then, by Proposition \ref{3.0},  $S$ is isomorphic either to $\mathbb{S}_5$ or $\mathbb{S}_6$. For the rest of the proof we therefore consider that $|S|\geq 3$. Let $w\in S$ be the greatest element in $S$ with respect to the addition.

 Assume, first, that $w$ is neither left nor right multiplicatively absorbing. By Theorem \ref{classification}(i), there is a finite (semi)lattice $L$ (with the least element $0$ and the greatest element $1$) such that $S$ is isomorphic to a subsemiring $R$ of $End_{0}(L)$ and for every $a,b\in L$ there is $e_{a,b}\in R$ such that for every $x\in L$ is $e_{a,b}(x)=0$ if $x\leq a$ and $e_{a,b}(x)=b$ otherwise. Assuming that there is $a\in L$ such that $0\neq a\neq 1$ we obtain that $0=e_{a,a}(a)=e_{a,a}(e_{a,a}(1))=e_{a,a}(1)=a$, a contradiction. Hence $|L|=2$ and therefore $|End_{0}(L)|=2$, by Proposition \ref{3.2}. Thus $|S|=2$ and we obtain a contradiction with the assumption that $|S|\geq 3$.

 Assume further, that $w$ is left but not right multiplicatively absorbing.  By Theorem \ref{classification}(ii), there is a finite semilattice $L$ (with the greatest element $1$) such that $S$ is isomorphic to a subsemiring $R$ of $End(L)$ and for every $a\in L$ and $b\in L\setminus\{1\}$ there is $f\in R$ such that for every $x\in L$ is $f(x)=b$ if $x\leq a$ and $f(x)>b$ otherwise. If there are $a\in L\setminus\{1\}$ and $b\in L\setminus\{1\}$ such that $a\not\leq b$ then we obtain that $b<f(b)=f(f(a))=f(a)=b$, a contradiction. Hence $|L|=2$ and $|End(L)|=3$, by Proposition \ref{3.2}. Therefore $S\cong End(L)=\mathbb{S}_7$.
 
 Finally, assume that $w$ is right but not left multiplicatively absorbing. Then the semiring $S^{op}$ is finite and congruence-simple. Further, $S^{op}$ has a greatest element (with respect to the addition) that is left but not right multiplicatively absorbing (and $|S^{op}|\geq 3$). By the previous part of the proof, we have that $S^{op}\cong \mathbb{S}_7$. Hence $S\cong \mathbb{S}_7^{op}=\mathbb{S}_8$.
\end{proof}

\begin{theorem}\label{3.3}
 Let $S$ be a finite multiplicatively idempotent congruence-simple semiring. Then $S$ is isomorphic to one of the eight two/three-element semirings $\mathbb{S}_1,\dots,\mathbb{S}_8$.
\end{theorem}
\begin{proof}
Let $|S|=2$. Then, by Propositions \ref{2.1.1} and \ref{3.0}, $S$ is isomorphic to one of the six semirings $\mathbb{S}_1,\dots,\mathbb{S}_6$.

Assume that $|S|\geq 3$. Then, by Propositions \ref{3.1}, $S$ is bi-idempotent. As $S$ finite, there is a greatest element $w$ with respect to the addition. If $w$ is multiplicatively absorbing then, by Theorem \ref{2.9}, $|S|=2$, a contradiction. Hence $w$ is not multiplicatively absorbing and therefore, by Proposition \ref{non-bi-absorbing},  $S$ is isomorphic to either $\mathbb{S}_7$ or $\mathbb{S}_8$.
%
%
%
%
%
%
\end{proof}

 With regard to Remark \ref{remark_5} and Theorems \ref{2.9} and \ref{3.3}, it is natural to consider the following conjecture:

\begin{conjecture}
Every multiplicatively idempotent congruence-simple semiring is finite (and isomorphic to one of the semirings $\mathbb{S}_1$, $\dots$, $\mathbb{S}_8$).
\end{conjecture}



\section{Multiplicatively divisible commutative semirings}

In \cite{parasem,conj} additively divisible commutative semirings were studied. Analogous questions may be raised for the multiplicative parts of commutative semirings.

 We call a semiring $S$ \emph{multiplicatively divisible} if for every $a\in S$ and every $n\in\N$ there is $b\in S$ such that $a=b^n$. A semiring $P$ is called a \emph{parasemifield} if the multiplicative semigroup $P(\cdot)$ is a group.

Clearly, every multiplicatively idempotent semiring is multiplicatively divisible. On the other hand, any algebraically closed field of characteristic $0$ (e.g., the countable field of algebraic real numbers) is congruence-simple and ideal-simple and both additively and multiplicatively divisible, but it is not multiplicatively idempotent.  Of course, no infinite field is finitely generated as a (semi)ring. The following conjecture is now in force.

\begin{conjecture}\label{conj_2}
Let $S$ be a finitely generated commutative semiring. If $S$ is multiplicatively divisible then $S$ is multiplicatively idempotent.
\end{conjecture}

The following Remarks \ref{remark_3} and \ref{remark_4} support the plausibility of this conjecture.

\begin{remark}\label{semigroup}
 Let $S$ be a divisible semigroup. Then $S$ is a finite band in each of the following two cases:
 
 \begin{enumerate}
  \item[(1)] $S$ is commutative and finitely generated. 
  \item[(2)] $S$ is finite.
 \end{enumerate}
This result is a sort of folklore (see, e.g., \cite[Theorem 2.5]{szekely}).
\end{remark}

\begin{remark}\label{remark_3}
Every finitely generated multiplicatively divisible commutative ring is a Boolean ring. Hence Conjecture \ref{conj_2} holds in case of rings.

Indeed, let $R$ be a finitely generated commutative ring (not necessary with a unity).
First, let us recall  a well-known fact (attributed sometimes to Kaplansky) that every field $F$ that is a factor of $R$ is finite. Let $R$ be multiplicatively divisible. Then, by Theorem \ref{semigroup}, the only such a field $F$ is $\Z_2$. Let $\mathcal{J}(R)$ be the Jacobson radical of $R$. Then $R/\mathcal{J}(R)$ is a subring of product of fields that are factors of $R$, i.e., of fields isomorphic to $\Z_2$. Such a product is multiplicatively idempotent and so is the ring $R/\mathcal{J}(R)$. 

To prove that $R$ is multiplicatively idempotent too, it is enough to show that $\mathcal{J}(R)=0$. Since the ring of all integers $\Z$ is a Jacobson ring (i.e., every prime ideal is an intersection of maximal ideals) and $R$ is a finitely generated $\Z$-algebra, the ring $R$ is Jacobson too. Hence, the Jacobson radical and the nilradical of $R$ coincide, by the definition of the Jacobson ring. Finally, as $R$ is also noetherian, the nilradical of $R$ is nilpotent. Thus, there is $n\in\N$ such that $\big(\mathcal{J}(R)\big)^n=0$. Further, $\mathcal{J}(R)$ is a radical ideal and therefore the ring $\mathcal{J}(R)$ is multiplicatively divisible. Hence, for every $a\in \mathcal{J}(R)$ there is $b\in \mathcal{J}(R)$ such that $a=b^n\in \big(\mathcal{J}(R)\big)^n=0$. This proves that $\mathcal{J}(R)=0$ and that $R$ is multiplicatively idempotent.
\end{remark}

\begin{remark}\label{remark_4}
According to \cite[Corollary 4.6]{parasem} if $S$ is a commutative  parasemifield, that is finitely generated as a semiring, then the multiplicative semigroup $S(\cdot)$ is finitely generated. Combining this result and  \ref{semigroup} we see that the semigroup $S(\cdot)$ is idempotent. In fact, as $S(\cdot)$ is a group, the only such a parasemifield $S$ is the trivial one. Hence Conjecture \ref{conj_2} holds in case of parasemifields.
\end{remark}


\end{document}